\documentclass{article}
\usepackage{fullpage}
\usepackage[round]{natbib}

\usepackage{multirow}
\usepackage{booktabs}
\usepackage{hyperref}
\usepackage{subcaption}

\usepackage{tikz}
\usetikzlibrary{shapes.multipart, backgrounds,calc}
\usetikzlibrary{babel, arrows, decorations, decorations.pathmorphing, decorations.pathreplacing, decorations.shapes}
\usetikzlibrary{arrows.meta, positioning}
\usetikzlibrary{fit}
\usetikzlibrary{shapes.geometric}
\tikzset{
    -Latex,auto,node distance =1 cm and 1 cm,semithick,
    state/.style ={circle, draw, minimum width = 0.85 cm},
    point/.style = {circle, draw, inner sep=0.04cm,fill,node contents={}},
    bidirected/.style={Latex-Latex,dashed},
    el/.style = {inner sep=2pt, align=left, sloped}
}

\usepackage{amsmath, amssymb, bm, amsthm}
\theoremstyle{plain} 
\newtheorem{theorem}{Theorem}
\newtheorem{proposition}{Proposition}

\theoremstyle{definition} 
\newtheorem{definition}{Definition}

\theoremstyle{remark} 

\usepackage[ruled,vlined]{algorithm2e}

\usepackage{authblk}

\title{SCMD: A Kernel-Based Distance for Structural Causal Models to Quantify Transferability Across Environments}

\author[1]{Théotime Le Goff}
\author[1]{Émilie Devijver}
\affil[1]{Univ Grenoble Alpes, CNRS, Grenoble INP, LIG}

\begin{document}

%

%

\maketitle

\begin{abstract}
Out-of-distribution generalization is key to building models that remain reliable across diverse environments. Recent causality-based methods address this challenge by learning invariant causal relationships in the underlying data-generating process. Yet, measuring how causal structures differ across environments, and the resulting generalization difficulty, remains difficult. To tackle this challenge, we propose the Structural Causal Model Distance (SCMD), a principled metric that quantifies discrepancies between two SCMs by combining (i) kernel-based distances for nonparametric comparison of distributions and (ii) pairwise interventional comparisons to capture differences in causal effects. We show that SCMD is a proper metric and provide a consistent estimator with theoretical guarantees. Experiments on synthetic and real-world datasets demonstrate that SCMD effectively captures both structural and distributional differences between SCMs, providing a practical tool to assess causal transferability and generalization difficulty.
\end{abstract}

\section{Introduction}

One of the major challenges in machine learning lies in the ability of models to generalize beyond the data observed during training. In domain adaptation, the objective is to transfer knowledge  from one or several source environments to a  different target environment \citep{ben2010theory, ganin2016domain, zhou2022domain}. The difficulty arises from the fact that most models primarily capture statistical correlations, which may prove unstable or irrelevant across environments \citep{beery2018recognition}. Beyond the performance drop, this reliance on correlations also raises concerns about interpretability, as it becomes difficult to understand why a model makes a given decision.

A promising alternative is to ground learning in the underlying causal structure of the data-generating process. Unlike correlations, causal mechanisms are assumed to remain invariant across different environments, providing a robust foundation for out-of-distribution prediction. This principle has been leveraged in different ways: \citet{peters2016causal} established causal inference in the linear case, later extended to nonlinear settings by \citet{heinze2018invariant}; \citet{rojas2018invariant} applied invariance to domain generalization, while \citet{magliacane2018domain} developped causal feature selection for domain adaptation. More recently, causal representation learning has been explored as a tool for extrapolation across distributions \citep{vonkügelgen2025representationlearningdistributionalperturbation}.
Together, these developments underline the importance of explicitly reasoning about causal structures when tackling generalization, but this raises two fundamental questions: (i) how to assess the difficulty of generalization from a set of training environments to a target one, and (ii) how to detect redundant information arising when multiple environments provide overlapping insights. 

Existing approaches are limited in this regard. Distributional distances \citep{gretton2012kernel, muandet2013domain} capture only part of the causal picture. Graph-based distances, such as Hamming distance or the Structural Intervention Distance (SID, \citealp{peters2013structural}), quantify structural discrepancies but overlook distributional differences.  Other graph distance measures, typically designed for causal discovery rather than transferability, have also been introduced: \cite{peyrard2020ladder} proposes a distance for each level of the causal ladder,  \cite{henckel2024adjustment} generalizes SID for class of graphs (instead of DAGs), and  \cite{wahl2025separationbaseddistancemeasurescausal}  goes further, extending the approach up to maximal ancestral graphs.
To combine structure- and distribution-aware knowledge, interventional KL divergences have recently been proposed \citep{Wildbergeretal23d} to evaluate causal inference and to quantify how good the model estimate is beyond observational.
In this work, we introduce a new metric, the Structural Causal Model Distance (SCMD), designed to directly compare structural causal models across environments. SCMD builds on kernel-based discrepancy measures to capture distributional shifts, while incorporating interventional comparisons inspired by the Structural Intervention Distance (SID) to reflect structural differences. In doing so, it provides the first principled tool to jointly quantify distributional and structural discrepancies between SCMs, thereby offering a practical proxy for assessing transferability.

The paper is organized as follows: Section \ref{sec:background} provides background and introduces the problem setup. Section \ref{sec:SCMD} presents our main contribution, the Structural Causal Model Distance. Section \ref{sec:est} develops a consistent estimator with theoretical guarantees. Section \ref{sec:exp} reports experiments on both synthetic and real-world data. Section \ref{sec:conc} concludes. Proofs, algorithmic details and additional experiments are provided in  Appendix.

\section{Background}\label{sec:background}
\subsection{Notations}

A \emph{Structural Causal Model} (SCM, \citet{pearl2000models}) is defined as $\mathcal{M}:=\left< \bm{{U}}, \bm{{V}}, \mathcal{F}, {P}({\bm{{U}}}) \right>$  where $\bm{U}$ is a set of exogenous variables taking values in $\mathcal{U}$, $\bm{V} = \{V_1,\ldots, V_d\}$ is a set of endogenous variables taking values in $\mathcal{V} = \mathcal{V}_1\times \ldots \times \mathcal{V}_d$, $\mathcal{F} = \{f_1,\ldots, f_d\}$ is a set of functions such that for each $1\leq j \leq d$, 
 $   f_j :
    ({pa}_j, {u}_j)\in (\mathcal{V}_{\mathrm{Pa}(V_j)} \times \mathcal{U}_j )  \mapsto v_j \in \mathcal{V}_j
$
and ${P}(\bm{U})$ is a probability distribution over mutually independent exogenous variables $\bm{U}$, where ${\mathrm{Pa}}(V_j)$ is the set of parents of $V_j$ and $pa_j$ its realization. 
The SCM $\mathcal{M}$ induces a \emph{directed acyclic graph} (DAG) $\mathcal{G}=(\bm{V},\mathcal{E})$ with an edge from $V_i$ to $V_j$ whenever $V_i\!\in\!\mathrm{Pa}(V_j)$. An \emph{intervention} $\mathrm{do}(V_k=v_k)$ replaces $f_k$ by the constant $v_k$, yielding a modified SCM and the corresponding {joint} interventional distribution $P_{\mathrm{do}(V_k=v_k)}(\bm{V})$. 

We assume throughout that the causal Markov condition and causal sufficiency (i.e., no hidden confounders) hold. 

\subsection{RKHS Embeddings and kernel-based distances} 
Comparing SCMs requires measuring discrepancies not only between their joint observational distributions, but also between their conditional and interventional distributions. Kernel methods provide a  framework for this comparison, as they allow us to embed distributions in a reproducing kernel Hilbert space (RKHS) and define metrics such as the Maximum Mean Discrepancy (MMD) and its conditional/interventional variants.

Let $(\mathcal{H}_{\mathcal{V}_j}, \langle \cdot, \cdot \rangle)$ be a reproducing kernel Hilbert space (RKHS, \cite{berlinet2011reproducing}) of real-valued functions on $\mathcal{V}_j$ with reproducing kernel $k_{\mathcal{V}_j} : \mathcal{V}_j \times \mathcal{V}_j \to \mathbb{R}$. The kernel mean embedding of a random variable $V_j \in \mathcal{V}_j$ with distribution $P(V_j)$ is defined as
\[
\mu_{P({V_j})}(\cdot) := \mathbb{E}_{V_j \sim {P}(V_j)}[k_{\mathcal{V}_j}(V_j, \cdot)] \in \mathcal{H}_{\mathcal{V}_j}.
\]
Given two joint distributions $P^1(V_j)$ and $P^2(V_j)$, the \emph{Maximum Mean Discrepancy} (MMD, \cite{gretton2012kernel}) is
\[
\mathrm{MMD}(P^1(V_j),P^2(V_j)) := \|\mu_{P^1(V_j)}(\cdot) - \mu_{P^2 (V_j)}(\cdot)\|_{\mathcal{H}_{\mathcal{V}_j}},
\]
which defines a metric between distributions for characteristic kernels \citep{fukumizu2007kernel}.

Kernel conditional mean embeddings \citep{park2020measure} represent conditional expectation operators in an RKHS $\mathcal{H}_{\mathcal{V}_j}$, which allows us  to define the \emph{Maximum Conditional Mean Discrepancy}  between two conditional distributions $P^2(V_j \mid V_i = v_i)$ and $P^2(V_j \mid V_i = v_i)$ by, {for $v_i \in \mathcal{V}_i$},
\begin{align*}
   & \mathrm{MCMD}(P^1(V_j \mid V_i = v_i), P^2(V_j \mid V_i = v_i))\\
    := &\|\mu_{P^1(V_j \mid V_i = v_i)}(\cdot) - \mu_{P^2(V_j \mid V_i = v_i)}(\cdot)\|_{\mathcal{H}_{\mathcal{V}_j}}.
\end{align*}

For interventional distributions, the adjustment formula is used to express them in terms of conditionals, which can be estimated from observational data.
For some $\bm{Z}$ \citep{perkovic2015}, a set of variables we can adjust for,
$$P_{\mathrm{do}(V_i = v_i)}(V_j) = \mathbb{E}_{\bm{Z} \sim P(\bm{Z})}[{P(V_j \mid V_i = v_i, \bm{Z})}],$$
with the following conventions:
\[
{P_{\mathrm{do}(V_i = v_i)}(V_j)} =
\begin{cases}
    {P(V_j \mid V_i=v_i)} & \text{ if } \bm{Z} = \emptyset,\\
    {P(V_j)} & \text{ if } V_i \not\to V_j,
\end{cases}
\]
where $V_i \not\to V_j$ denotes the absence of any path, whether direct or indirect, from $V_i$ to $V_j$. Then the interventional mean embedding of $V_j$ under the intervention $\mathrm{do}(V_i = v_i)$ is defined as in \citet{dhanakshirur2023continuous},
\begin{align}
    &\mu_{P_{\mathrm{do}(V_i = v_i)}(V_j)}(\cdot) \label{eq:cases}\\ 
    := & 
\begin{cases}
 \mu_{P(V_j)}(\cdot)  &\text{ if } V_i \not\to V_j, \nonumber\\
  \mu_{P(V_j\mid V_i=v_i)}(\cdot)  &\text{ if } \bm{Z} = \emptyset,\nonumber\\
\mathbb{E}_{\bm{Z} \sim P(\bm{Z})}[\mu_{P(V_j \mid V_i = v_i, \bm{Z}=\bm{z})}(\cdot)] &\text{ elsewhere.}\nonumber
\end{cases}
\end{align}

The \emph{Maximum Interventional Mean Discrepancy}  between two interventional distributions is given by:
\begin{align*}
    &\mathrm{MIMD}(P^1_{\mathrm{do}(V_i=v_i)}(V_j), P^2_{\mathrm{do}(V_i = v_i)}(V_j))\nonumber\\
    := &\|\mu_{P^1_{\mathrm{do}(V_i = v_i)}(V_j)}(\cdot) - \mu_{P^2_{\mathrm{do}(V_i = v_i)}(V_j)}(\cdot)\|_{\mathcal{H}_{\mathcal{V}_j}}.
\end{align*}
This can be seen as the MCMD  between two expectations of conditional embeddings.

\subsection{Estimation of RKHS embeddings}
Let $\mathcal{D} = (v^{(n)}_j)_{1\leq j \leq d, 1\leq n \leq N}$ be a sample, generated from an SCM $\mathcal{M}$. While the above embeddings are defined in population, in practice we only have access to finite samples. We now describe how to estimate the embeddings consistently from data, which will be crucial for our empirical evaluation of SCMD. For more details on estimation of RKHS embeddings, see \cite{surveyMuandet}. 

\paragraph{Marginal embeddings}  
The kernel mean embedding is estimated empirically as in \cite{gretton2012kernel}:
\[
\hat{\mu}^N_{P(V_j)}(\cdot) = \frac{1}{N} \sum_{n=1}^N k_{\mathcal{V}_j}(v_j^{(n)}, \cdot).
\]
This estimator is consistent, and converges with a rate of $O(N^{-1/2})$.
\paragraph{Conditional embeddings}  
Following \cite{park2020measure}, we estimate the conditional kernel mean embedding via regularized RKHS regression: for a regularization parameter $\lambda \geq 0$, {for $v_i \in \mathcal{V}_i$},
\[
\hat{\mu}^{\lambda, N} _{P(V_j \mid V_i=v_i)}(\cdot) = \bm{k}_{V_i}^\top(v_i) \, \bm{W}^\lambda_{V_i} \, \bm{k}_{V_j} (\cdot),
\]
where
\[
\bm{k}_{V_i}(v_i) = \begin{pmatrix} k_{\mathcal{V}_i}(v_i^{(1)},v_i) \\ \vdots \\ k_{\mathcal{V}_i}(v_i^{(N)},v_i) \end{pmatrix}, 
\quad 
\bm{W}^\lambda_{V_i} = (\bm{K}_{V_i} + N \lambda I_N)^{-1},
\]
and $[\bm{K}_{V_i}]_{st} = k_{\mathcal{V}_i}(v_i^{(s)}, v_i^{(t)})$  is the Gram matrix comprising all pairwise similarity values between the observations of the sample from $V_i$.
This estimator is consistent, and converges with a rate of $O(N^{-1/4})$.
\paragraph{Interventional embeddings}  
For  interventions with an adjustment set $\bm{Z}$, the interventional kernel mean embedding can be expressed as
\[
\mu_{P_{\mathrm{do}(V_i=v_i)}(V_j)}(\cdot) = \mathbb{E}_{\bm{Z} \sim P(\bm{Z})}[\mu_{P(V_j \mid V_i=v_i, \bm{Z}=\bm{z})}(\cdot)].
\]  
We use the  estimator by \cite{dhanakshirur2023continuous}:
\begin{align*}
    \hat{\mu}^{\lambda,N}_{P_{\mathrm{do}(V_i=v_i)}(V_j)}(\cdot) &= \frac{1}{N} \sum_{n=1}^N \hat{\mu}_{P(V_j \mid V_i=v_i,\bm{Z}=\bm{z}^{(n)})}(\cdot)\\
    &\hspace{-2cm}= \left(\frac{1}{N} \sum_{n=1}^N   \bm{k}_{V_i, \bm{Z}}^\top(v_i,\bm{z}^{(n)})\right) \bm{W}^\lambda_{V_i, \bm{Z}} \, \bm{k}_{V_j}(\cdot),
\end{align*}
where $\bm{k}_{V_i, \bm{Z}}^\top(v_i,\bm{z}):= \bm{k}_{V_i}^\top(v_i)\odot\bm{k}_{\bm{Z}}^\top(\bm{z})$ is a vector of size $N$, with $\odot$ the Hadamard product, {and $\bm{W}^\lambda_{V_i, \bm{Z}} = (\bm{K}_{V_i,\bm{Z}} + N \lambda I_N)^{-1}$ where $\bm{K}_{V_i,\bm{Z}}$ is also constructed with the Hadamard product}.


\subsection{Problem setup}
We consider two SCMs $\mathcal{M}^1$ and $\mathcal{M}^2$ defined on the same variables $\bm{V}$, inducing graphs $\mathcal{G}^1=(\bm{V},\mathcal{E}^1)$ and $\mathcal{G}^2=(\bm{V},\mathcal{E}^2)$ and joint distributions ${P}^1({\bm{V}})$ and ${P}^2({\bm{V}})$. 
Let $\mathcal{D}^1$ and $\mathcal{D}^2$ be two datasets sampled from the observational distributions entailed by $\mathcal{M}^1$ and $\mathcal{M}^2$ respectively. 

{Our objective is to define and estimate a distance between $\mathcal{M}^1$ and $\mathcal{M}^2$ that captures discrepancies in both  causal structure and  interventional behaviour.}

\section{SCMD: Measuring Structural and Distributional Differences Between Structural Causal Models} \label{sec:SCMD}
\subsection{The Structural Causal Model Distance}
We introduce  the Structural Causal Model Distance (SCMD), a distance between two SCMs $\mathcal{M}^1$ and $\mathcal{M}^2$ that quantifies how their induced distributions evolve under interventions. Intuitively, SCMD measures the discrepancy between the interventional distributions induced by two SCMs, for all possible pairs of intervention-target variables. 
This captures both structural differences (e.g., reversed causal edges) and distributional differences (e.g., changes in noise or functional relationships).

\begin{definition}
Let $\mathcal{M}^1,\mathcal{M}^2$ two SCMs and $(\bm{v}^1,\bm{v}^2)$ two vectors of intervention values. We define the \emph{Structural Causal Model Distance} (SCMD) by 
\begin{align}
&\mathrm{SCMD}(\mathcal{M}^1, {\mathcal{M}}^2 ; \bm{v}^1,\bm{v}^2) \label{SCMD}\\
:=& \sum_{1\leq i,j \leq d}
\mathrm{MIMD} \left({P}_{\mathrm{do}(V_i = v_i^1)}^1(V_j),
{P}_{\mathrm{do}(V_i = v_i^2)}^2(V_j)\right) \nonumber\\
= & \sum_{1\leq i,j \leq d} \|\mu_{P^1_{\mathrm{do}(V_i = v_i^1)}(V_j)}(\cdot) - \mu_{P^2_{\mathrm{do}(V_i = v_i^2)}(V_j)}(\cdot)\|_{\mathcal{H}_{\mathcal{V}_j}}. \nonumber
\end{align}
\end{definition}
In practice, we use the parent set $\mathrm{Pa}(V_i)$ as the adjustment set for estimating interventional distributions $P_{\mathrm{do}(V_i = v_i)}(V_j)$, which ensures identifiability when the causal graph is known. However, other adjustment sets could be considered for improved estimation efficiency (see, e.g., \citet{NEURIPS2021_8485ae38}).
Two vectors of interventions have to be specified: if only one value is set, the vector is the same for the two environments. 

\paragraph{Prediction-oriented variant.}
In many applications, only the effect of interventions on a specific outcome variable  is of interest. We thus introduce a prediction-oriented variant, P-SCMD, which restricts the comparison to interventions affecting a target variable $V_1$:
\begin{align*}
&\mathrm{P\mbox{-}SCMD}(\mathcal{M}^1, {\mathcal{M}}^2 ; (v_i^1)_{2\leq i \leq d}, (v_i^2)_{2\leq i \leq d})) \\
:=& \sum_{2\leq i \leq d}
\mathrm{MIMD}\left({P}_{\mathrm{do}(V_i = v_i^1)}^1({V_1}),
{P}^2_{\mathrm{do}(V_i = v_i^2)}({V_1})\right). 
\end{align*}
This variant provides a task-specific comparison between models, focusing on the causal effect of interventions on the prediction target.

\paragraph{Choice of intervention values.}
Since the choice of intervention values $(\bm{v}^1, \bm{v}^2)$ may be arbitrary, we also define the expected SCMD (E-SCMD), which averages the distance over the joint distribution of $\bm{V}$. This provides a global measure of discrepancy, independent of specific intervention values.
\begin{align*}
&\mathrm{E\mbox{-}SCMD}(\mathcal{M}^1, {\mathcal{M}}^2) \\
:= &\mathbb{E}_{\bm{V}^1, \bm{V}^2}
\bigl[\mathrm{SCMD}(\mathcal{M}^1, {\mathcal{M}}^2 ; \bm{V}^1, \bm{V}^2)\bigr].
\end{align*}
Other choices are possible, for instance focusing on a grid of intervention values or on application-specific points of interest.

\subsection{Properties of SCMD}
We now establish the theoretical properties of SCMD, which justify its use as a metric for comparing SCMs. In particular, we show that SCMD is a proper distance, bounded, and related to existing metrics such as SID. Proofs are provided in the Appendix.


Since SCMD is defined as a sum of MIMD terms, each of which is a valid distance in an RKHS, it inherits their metric properties.

\begin{proposition}[SCMD is a distance]
Let $\mathcal{M}^1,\mathcal{M}^2$ two SCMs and $(\bm{v}^1, \bm{v}^2)$ two vectors of intervention values. If the kernel $k_{\mathcal{V}}$ is characteristic, then $\mathrm{SCMD}$ defined in Equation \eqref{SCMD} is a distance: 
    \begin{itemize}
        \item  \textbf{Separation:} $\mathrm{SCMD}(\mathcal{M}^1,\mathcal{M}^2 ; \bm{v}^1, \bm{v}^2) = 0\\ \Leftrightarrow \mathcal{M}^1 = \mathcal{M}^2$ almost surely;
        \item  \textbf{Non-negativity:}  $\mathrm{SCMD}(\mathcal{M}^1,\mathcal{M}^2; \bm{v}^1, \bm{v}^2) \geq 0$;
        \item \textbf{Symmetry:} $$\mathrm{SCMD}(\mathcal{M}^1,\mathcal{M}^2; \bm{v}^1, \bm{v}^2) = \mathrm{SCMD}(\mathcal{M}^2,\mathcal{M}^1; \bm{v}^2, \bm{v}^1);$$
        \item \textbf{Triangle inequality:} for any third SCM $\mathcal{M}_3$,  \begin{align*}
            &\mathrm{SCMD}(\mathcal{M}^1,\mathcal{M}^2; \bm{v}^1, \bm{v}^2)\\
            \leq &\mathrm{SCMD}(\mathcal{M}^1,\mathcal{M}^3; \bm{v}^1, \bm{v}^3) +\mathrm{SCMD}(\mathcal{M}^3,\mathcal{M}^2; \bm{v}^3, \bm{v}^2).
        \end{align*}
    \end{itemize}
\end{proposition}

In contrast with SID \citep{peters2013structural}, SCMD is separable, allowing to test whether two SCMs are identical, and symmetric, which is natural for model comparison.

In the next proposition, we go further by bounding the SCMD.

\begin{proposition}[Bounds and relation to SID]
Let $\mathcal{M}^1,\mathcal{M}^2$ two SCMs and $\bm{v}^1, \bm{v}^2$ two vectors of intervention values. 
\begin{enumerate}
\item \textbf{Boundedness:} if the kernel $k_{\mathcal{V}}$ is bounded, then SCMD is bounded as well:
     if for all $(v_j, v'_j) \in \mathcal{V}_j^2$, $k_{\mathcal{V}_j}(v_j,v'_j)\leq K$ for $K\in \mathbb{R}^+$, then 
 $$\mathrm{SCMD}(\mathcal{M}^1,\mathcal{M}^2 ; \bm{v}^1, \bm{v}^2)\leq \sqrt{2K}d(d-1).$$

\item \textbf{Relation to SID:} 
\begin{align*}
   & \mathrm{SCMD}(\mathcal{M}^1,\mathcal{M}^2; \bm{v}^1, \bm{v}^2)=0\\
    &\Rightarrow 
     \mathrm{SID}(\mathcal{G}^1,\mathcal{G}^2)= \mathrm{SID}(\mathcal{G}^2,\mathcal{G}^1)=0. 
\end{align*}
\end{enumerate}
\end{proposition}

The boundedness of SCMD follows from the boundedness of the kernel $k_{\mathcal{V}_j}$, ensuring that each MIMD term is at most $\sqrt{2K}$. This property is useful for normalization and theoretical analysis.
The second implication  shows that SCMD is at least as discriminative as SID: if two SCMs are identical in terms of SCMD, they must also share the same causal graph. However, the converse is not true, as SCMD also captures distributional differences within the same graph.

\subsection{Illustrative Example} \label{sec:ex}
We now illustrate how SCMD captures both parametric (same causal graph, different joint distributions) and structural (different causal graphs, same joint distribution) differences between SCMs, using a simple linear Gaussian example. This example also highlights the limitations of existing metrics (MMD, SID) and the complementary strengths of SCMD.

Let $a\neq 0$, and consider two SCMs:

\begin{align*}
    &\mathcal{M}^{1,a}:\begin{cases} X \sim \mathcal{N}(0,\,1)\\[0.5ex] 
    Y = a X + \varepsilon_Y, \text{ where } \varepsilon_Y \sim \mathcal{N}(0,\,1); \end{cases} \\
    &\mathcal{M}^{2,a}:\begin{cases} Y \sim \mathcal{N}\!\left(0,\,1+a^2\right)\\[0.5ex] X = \tfrac{a}{1+a^2} Y + \varepsilon_X, \text{ where } \varepsilon_X \sim \mathcal{N}\!\left(0,\,\tfrac{1}{1+a^2}\right). \end{cases}
\end{align*}
They lead to the following graphs, respectively:
\begin{center}
\begin{tikzpicture}[state/.style={circle, draw, minimum size=6mm}]
    \node[state] (X1) at (0,0) {$X$};
    \node[state] (Y1) at (1.5,0) {$Y$};
    \path (X1) edge (Y1);
    \node[draw=none] (a1) at (-0.9,0) {$\mathcal{G}^1$:};
    
    \node[state] (X2) at (4,0) {$X$};
    \node[state] (Y2) at (5.5,0) {$Y$};
    \path (Y2) edge (X2);
    \node[draw=none] (a2) at (3.1,0) {$\mathcal{G}^2$:};
\end{tikzpicture}
\end{center}

Despite the reversed direction of causality, both SCMs induce the same joint Gaussian distribution :
$$P^{1,a} = P^{2,a} = 
\mathcal{N}\!\left(
\begin{pmatrix}
    0 \\
    0
\end{pmatrix},
\begin{pmatrix}
    1 & a \\
    a & 1+a^2
\end{pmatrix}\right).$$

Let $(\mathcal{H}= \{f : \mathbb{R} \rightarrow \mathbb{R} \}, \langle \cdot, \cdot \rangle)$ an RKHS, with  the Gaussian kernel $k$, parameterized by a variance $\sigma^2$. Let $\bm{v}^1=\bm{v}^2=(x,y)$ the value of the interventions we consider. 

\paragraph{Case 1: Parametric Differences (Same Graph, Different Parameters)}  Consider comparing $\mathcal{M}^{1,a}$ and $\mathcal{M}^{1,b}$ with $a \neq b \neq 0$. Here, the causal graph is the same, but the functional relationship between $X$ and $Y$ differs.
MMD between the joint distributions $P^{1,a}$ and $P^{1,b}$ is non-zero, as the joint distributions differ. However, SID between $\mathcal{G}^1$ and $\mathcal{G}^1$ is zero, as the graphs are identical. SCMD, in contrast, detects the parametric difference between the two models:
\begin{align*}
  & \text{SCMD}(\mathcal{M}^{1,a}, \mathcal{M}^{1,b}; (x, y) )\\
    = &\|\mu_{\mathcal{N}(0,\,1)} -\mu_{\mathcal{N}(0,\,1)} \|_{\mathcal{H}} + \|\mu_{\mathcal{N}(a x,\,1)} -\mu_{\mathcal{N}(b x,\,1)} \|_{\mathcal{H}}\\
    =&  \left(2\sqrt{\frac{\sigma^2}{\sigma^2+2}} \left(1-\exp\left(-\frac{(a-b)^2 x^2}{2(\sigma^2+2)}\right)\right)\right)^{1/2}.
\end{align*}
This result shows that SCMD is sensitive to changes in the functional relationship, even when the causal structure remains unchanged, a limitation of SID.

\begin{figure}[t]
    \centering
    \includegraphics[width=0.5\columnwidth]{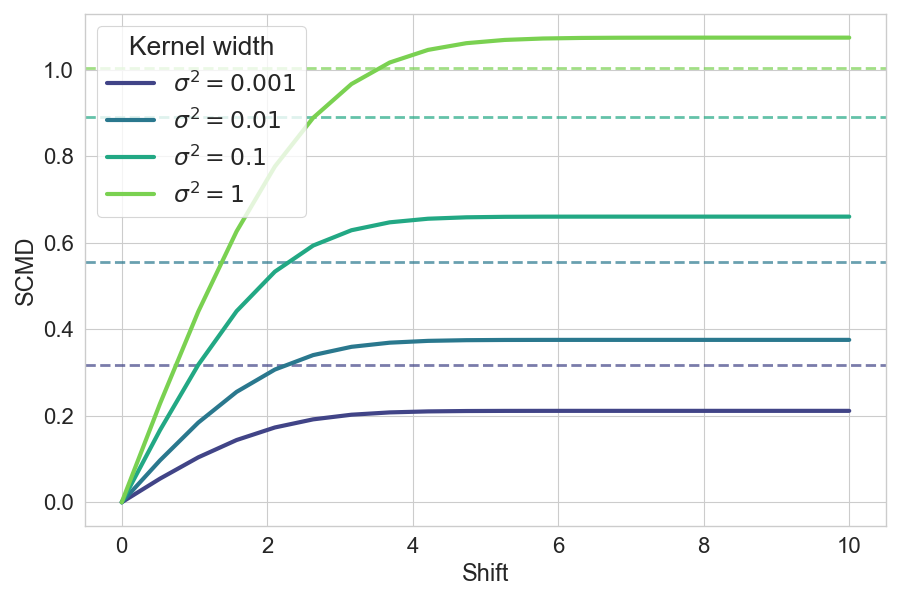}
    \caption{Comparison between $\mathrm{SCMD}(\mathcal{M}^{1,a},\mathcal{M}^{2,a})$ (dotted line)
    and $\mathrm{SCMD}(\mathcal{M}^{1,a},\mathcal{M}^{1,a+\text{shift}})$ (solid line), plotted as a function of the shift parameter. Colors correspond to different values of $\sigma^2$.}
    \label{fig:SCMDShift}
\end{figure}

\paragraph{Case 2: Structural Differences (Different Graphs, Same Joint Distribution)
}  
Now consider comparing $\mathcal{M}^{1,a}$ and $\mathcal{M}^{2,a}$. Here, the observational distribution is identical, but the causal direction is reversed.
MMD between the joint distributions is zero, as $P^{1,a} = P^{2,a}$.  SID between $\mathcal{G}_1$ and $\mathcal{G}_2$  is 2, reflecting the structural difference. SCMD also detects this structural difference:
\begin{align*}
   &\text{SCMD}(\mathcal{M}^{1,a}, \mathcal{M}^{2,a}; (x, y) )
    = \|\mu_{\mathcal{N}(a x,\,1)} - \mu_{\mathcal{N}\!\left(0,\,1+a^2\right)}\|_{\mathcal{H}} \\
    & \hspace{0.5cm} + \|\mu_{\mathcal{N}(0,\,1)} - \mu_{ \mathcal{N}\!\left(\tfrac{a}{1+a^2}y,\,\tfrac{1}{1+a^2}\right)}\|_{\mathcal{H}}
    \neq 0.
\end{align*}
The explicit formulae is given in appendix.
This demonstrates that SCMD, like SID, captures structural differences, but unlike MMD, it does so even when the observational distribution is unchanged. 

\paragraph{Discussion}
Figure~\ref{fig:SCMDShift}  illustrates the  behavior of SCMD in both cases, as a function of the shift parameter and the kernel bandwidth  $\sigma^2$. 
The relative magnitude of distributional shifts and structural differences depends on $\sigma^2$. For large $\sigma^2$, SCMD focuses on global structural differences, while for small $\sigma^2$, it becomes more sensitive to local distributional shifts. This flexibility allows SCMD to adapt to different analysis goals.

Together, these two cases demonstrate how SCMD interpolates between MMD and SID: it reacts to parametric changes within a fixed graph, like MMD, and captures structural differences that leave the observational distribution unchanged, like SID. Unlike both, it provides a unified metric for comparing SCMs in terms of their interventional behavior, making it a powerful tool for causal model comparison.

\section{Estimation and Guarantees} \label{sec:est}
Let $\mathcal{D}^1 = (v_{1,j}^{(n)})_{1\leq j \leq d ; 1\leq n \leq N}$ and  $\mathcal{D}^2 = (v_{2,j}^{(n)})_{1\leq j \leq d ; 1\leq n \leq N}$  be two samples, generated respectively from two SCMs $\mathcal{M}^1$ and $\mathcal{M}^2$ respectively. 
Throughout this part, we assume that the underlying causal graphs $\mathcal{G}^1$ and $\mathcal{G}^2$ are known. 

\subsection{Estimator Construction}
The SCMD introduced in Eq. \eqref{SCMD} is defined as a sum of distances between kernel-based interventional embeddings for all pairs of variables. Each embedding corresponds to a marginal, conditional, or interventional distribution, as specified in Section~\ref{sec:background}. We now describe an estimator of each embedding, and thus of the SCMD. 

For an effect of $V_i$ on $V_j$ computed with a given adjustment set $\mathbf{Z}$, let $\widetilde{\mathcal{H}}$ be a vector-valued RKHS of functions  $\mathcal{V}_i \rightarrow \mathcal{H}_{\mathcal{V}_j}$. 
The estimator of the (conditional/interventional) embedding adapts to the causal structure, similarly to the embedding defined in Eq. \eqref{eq:cases}: for a regularization parameter $\lambda \geq 0$, for $v_i \in \mathcal{V}_i$, 
\begin{equation*}
  \widehat{\mu}^{\lambda,N}_{P_{\text{do}(V_i = v_i)}(V_j)}(\cdot)
  =  \widehat{\Omega}^{\lambda,N}_{v_i} \,\bm{k}_{V_j} (\cdot),
  \end{equation*}
  where $\widehat{\Omega}_{v_i}^{\lambda,N}$ is defined as:
\begin{align}
\widehat{\Omega}^{\lambda,N}_{v_i} \!\!:= &\frac{1}{N} \mathbf{1}_N^\top 
    && \hspace{-0.1cm} \text{if } V_i \not\to V_j, \label{kw:sim}\\
    &\bm{k}^\top_{V_i}(v_i)  \bm{W}^\lambda_{V_i}
    && \text{if } \bm{Z} = \emptyset, \label{kw:cond}\\
    &\frac{1}{N}  \sum_{n=1}^N \bm{k}^\top_{V_i, \bm{Z}}(v_i,\bm{z}^{(n)}) \bm{W}^\lambda_{V_i, \bm{Z}}
    && \text{elsewhere.} \label{kw:int}
\end{align}

This estimator minimizes the regularized empirical loss 
\begin{align*}
\tilde{\mathcal{E}}_{N, \lambda}(\mu) = \frac{1}{N} \sum_{n=1}^N [\|k_{\mathcal{V}_j}(v_j^{(n)}, \cdot) - \mu(v_i^{(n)})\|^2_{\mathcal{H}_{V_j}} + \lambda \| \mu \|_{\widetilde{\mathcal{H}}},
\end{align*}
as introduced in \cite{park2020measure}, seen as a surrogate loss for the conditional mean embedding estimation. 



\subsection{Theoretical Guarantees}
The following theorem establishes the consistency of the proposed estimator.

\begin{theorem}
Suppose that $k_{\mathcal{V}_i}, k_{\mathcal{V}_j}$ and $k_{\mathcal{Z}}$ are bounded kernels, and that the operator-valued kernel $k_{\mathcal{V}\mathcal{Z}}$ is $C_0$ universal. Let the regularization parameter $\lambda$ decay to 0 at a slower rate than $\mathcal{O}(n^{-1/2})$. 
Then $\widehat{\mu}^{\lambda,N}_{P_{\text{do}(V_i = v_i)}(V_j)}$ is universally consistent.

Assume further that $\mu_{P_{\text{do}(V_i = v_i)}(V_j)}(\cdot) \in \widetilde{\mathcal{H}}$. Then, with probability $1-\delta$, 
$$\tilde{\mathcal{E}}_{N, \lambda}(\widehat{\mu}^{\lambda,N}_{P_{\text{do}(V_i = v_i)}(V_j)}) - \tilde{\mathcal{E}}_{ N, \lambda}(\mu_{P_{\text{do}(V_i = v_i)}(V_j)}) = \mathcal{O}_P (N^{-1/4}).$$
\end{theorem}

Remark that this rate of convergence is slower than the $\mathcal{O}(N^{-1/2})$ rate achieved for marginal embeddings, due to the additional complexity of conditional and interventional estimation.

The estimator of MIMD is then given by:
\begin{align*}
    &\widehat{\mathrm{MIMD}}\left({P}_{\mathrm{do}(V_i = v_i)}^1({V_j}),{P}^2_{\mathrm{do}(V_i = v_i)}({V_j})\right)\\
    = &\left( [\widehat{\Omega}^{\lambda, N,1}_{v_i, \mathbf{z}}]^\top \mathbf{K}^1  \widehat{\Omega}^{\lambda, N,1}_{v_i, \mathbf{z}} 
    + [\widehat{\Omega}^{\lambda, N,2}_{v_i, \mathbf{z}}]^\top \mathbf{K}^2  \widehat{\Omega}^{\lambda, N,2}_{v_i, \mathbf{z}} \right.\\
    & \left. -2 [\widehat{\Omega}^{\lambda, N,1}_{v_i, \mathbf{z}}]^\top \mathbf{K}^{1,2}  \widehat{\Omega}^{\lambda, N,2}_{v_i, \mathbf{z}}\right)^{1/2},
\end{align*}
where $\mathbf{K}^1, \mathbf{K}^2$, and $\mathbf{K}^{1,2}$ are the Gram matrices of the observations from $V_j$ in $\mathcal{D}^1, \mathcal{D}^2$, and between $\mathcal{D}^1$ and $\mathcal{D}^2$, respectively.

Then, since SCMD aggregates the distances between kernel-based interventional embeddings over all pairs of variables from two distinct SCMs $\mathcal{M}^1$ and $\mathcal{M}^2$, its plug-in estimator is also universally consistent.

\begin{algorithm}[t]
\caption{\texttt{Omega}}
\KwIn{$\mathcal{G}$,  $\mathcal D$, $(i,j)$ and $v_i$.  }
\KwOut{The vector $ \widehat{\Omega}^{\lambda,N}_{v_i}$}
\uIf{there is no path from $V_i$ to $V_j$ in $\mathcal{G}$}{
    \Return{Eq. \eqref{kw:sim}}\;
}
\uElse{
    \uIf{$\mathrm{Pa}(V_i)=\emptyset$}{
        \Return{Eq. \eqref{kw:cond}}\;
    }
    \uElse{
    $\mathbf{Z} = \mathrm{Pa}(V_i)$\;
        \Return{Eq. \eqref{kw:int}}\;
    }
}
\label{algo:kw}
\end{algorithm}

\subsection{SCMD Algorithm}
 The overall algorithm relies on the computation of $\widehat{\Omega}^{\lambda,N}_{v_i}$ (Algorithm~\ref{algo:kw}) and MIMD (Algorithm~\ref{algo:mimd}) computed over all pairs of variables $(i,j)$. 
All those algorithms suppose given a regularization parameter $\lambda \geq 0$, a kernel $k$ with bandwidth $\sigma^2$. 
The function $\texttt{Kernel}$  in  Algorithm \ref{algo:mimd}   computes the Gram matrix of all the observations of $\mathcal{D}$ from the variable $j$.



\begin{algorithm}[t]
\caption{\texttt{MIMD}}
\KwIn{ $\mathcal{G}^1$, $\mathcal{G}^2$,  $\mathcal D^1$, $\mathcal D^2$, $(i,j)$ and $v_i$.  }
\KwOut{$\widehat{\mathrm{MIMD}}\left({P}_{\mathrm{do}(V_i = v_i)}^1({V_j}),{P}^2_{\mathrm{do}(V_i = v_i)}({V_j})\right)$}
$\Omega^1$ $\leftarrow$ \texttt{Omega}($\mathcal{G}^1$, $D^1$, $i$, $j$, $v_i$)\;
$\Omega^2$ $\leftarrow$ \texttt{Omega}($\mathcal{G}^2$, $D^2$,  $i$, $j$, $v_i$)\;
$\bm K_{V_j}^1$ $\leftarrow$ \texttt{Kernel}($\mathcal{D}^1$, $j$)\;
$\bm K_{V_j}^2$ $\leftarrow$ \texttt{Kernel}($\mathcal{D}^2$, $j$)\;
$\bm K_{V_j}^{1,2}$ $\leftarrow$ \texttt{Kernel}($\mathcal{D}^1$, $\mathcal{D}^2$, $j$)\;
\Return{$(\Omega^{1\top} \bm K^1_{V_j} \Omega^1 - 2 \Omega^{1\top} \bm K_{V_j}^{1,2} \Omega^2 + \Omega^{2\top} \bm K^2_{V_j} \Omega^2)^{1/2}$}
\label{algo:mimd}
\end{algorithm}

The algorithmic complexity to compute SCMD over two datasets of $d$ variables observed on $N$ observations is  $\mathcal{O}(d^3N^3)$. Details are provided in the Appendix.

\section{Numerical Experiments} \label{sec:exp}
We evaluate the proposed SCMD metric through two sets of experiments: first, on a  synthetic dataset generated from the illustrative example in Section \ref{sec:ex}, to illustrate its behavior and build intuition; second, on a real-world benchmark composed of nine datasets collected under heterogeneous environments. 

Throughout all experiments, we use a Gaussian kernel, which is characteristic and only depends on the bandwidth parameter $\sigma^2$. We assume that the causal graph is known, allowing us to identify the appropriate adjustment set for each intervention.

Experiments are run on a standard desktop computer. Code for these experiments is included in the Supplementary Material for anonymity and will be made publicly available upon acceptance.

\begin{table}[t]
\centering
\caption{Results on the simulated example.  
Case 1: same causal graph but different joint distributions;  
Case 2: different causal graphs but same joint distribution. 
Reported values are means (standard deviation) over $50$ repetitions ($\widehat{\mathrm{E\text{-}SCMD}}$ over $20$ repetitions).}
\begin{tabular}{lcc}
\toprule
 & Case 1 & Case 2 \\ 
\midrule
SID & $0$ & $2$ \\ 
\midrule
MMD & $0.0515$ & $0$ \\
$\widehat{\mathrm{MMD}}_{\mathrm{kernel}}$ & $0.1215(\pm 0.0013)$ & $0.0140(\pm 0.0016)$\\
\midrule
SCMD  & $0.5177$ & $0.8921$ \\
$\widehat{\mathrm{SCMD}}_{\mathrm{plug-in}}$ & $0.5248(\pm 0.0047)$ & $0.8929(\pm 0.0038)$ \\
$\widehat{\mathrm{SCMD}}_{\mathrm{kernel}}$ & $0.5360(\pm 0.0173)$ & $1.0046(\pm 0.0169)$ \\ 
\midrule
$\widehat{\mathrm{P\text{-}SCMD}}_{Y}$ & $0.5237(\pm 0.0172)$ & $0.4062(\pm 0.0091)$ \\
$\widehat{\mathrm{P\text{-}SCMD}}_{X}$ & $0.0122(\pm 0.0033)$ & $0.5984(\pm 0.0145)$ \\
\midrule
$\widehat{\mathrm{E\text{-}SCMD}}$ & $0.5819(\pm 0.0060)$ & $0.9473(\pm 0.0059)$ \\
\bottomrule
\end{tabular}
\label{tab:res}
\end{table}
\subsection{Illustrative Example}

\paragraph{Data generating process}
We use the setup described in Section~\ref{sec:ex} with $N=10^4$ samples, comparing the two scenarios.
The parameters are set to $a=3$, $b=5$, and the intervention is fixed to $x=y=1$.
We use  $\sigma^2=0.1$ and $\lambda=0.5$ in the RKHS-valued regression. 
A sensitivity analysis with respect to these hyperparameters is provided in  Appendix.

\paragraph{Estimators}
For each case, we compute:
(i) the theoretical value of the SID between the causal graphs;
(ii) the theoretical value of the MMD between the joint distributions, together with its biased V-statistic estimator; 
(iii) the theoretical value of SCMD; together with its plug-in estimator, obtained by substituting estimated means and variances (possibly conditional) into the closed-form expression; and the fully kernel-based estimator proposed in this paper, (iv) its extension to P-SCMD and (v) E-SCMD (computed uniformly over quantiles from 0.2 to 0.8 level).

\begin{figure}[t]
    \centering
    \includegraphics[width=0.4\columnwidth, trim = 0 1.2cm 0 1.2cm, clip=TRUE]{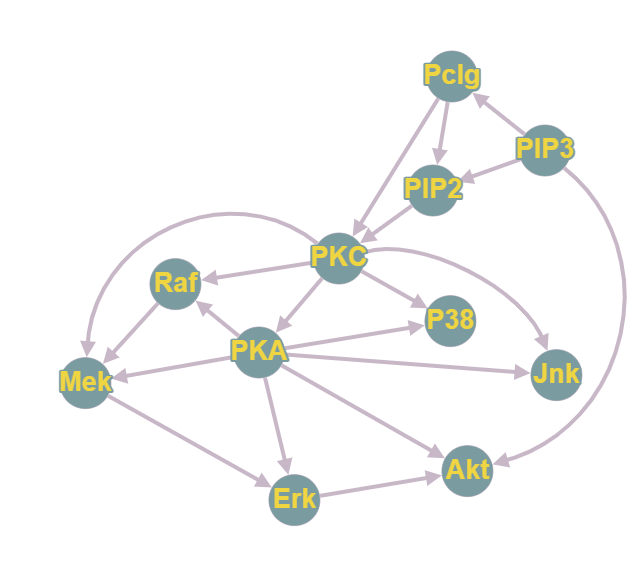}
    \caption{Causal graph of the real-world dataset \citep{sachs2005causal}. Nodes correspond to proteins or phospholipids, while edges were obtained through causal discovery and validated by experts.}
    \label{fig:Sachs}
\end{figure}

\begin{figure*}[t]
    \centering
    \includegraphics[width=0.85\textwidth, trim = 0 1.5cm 0 0.9cm, clip = TRUE]{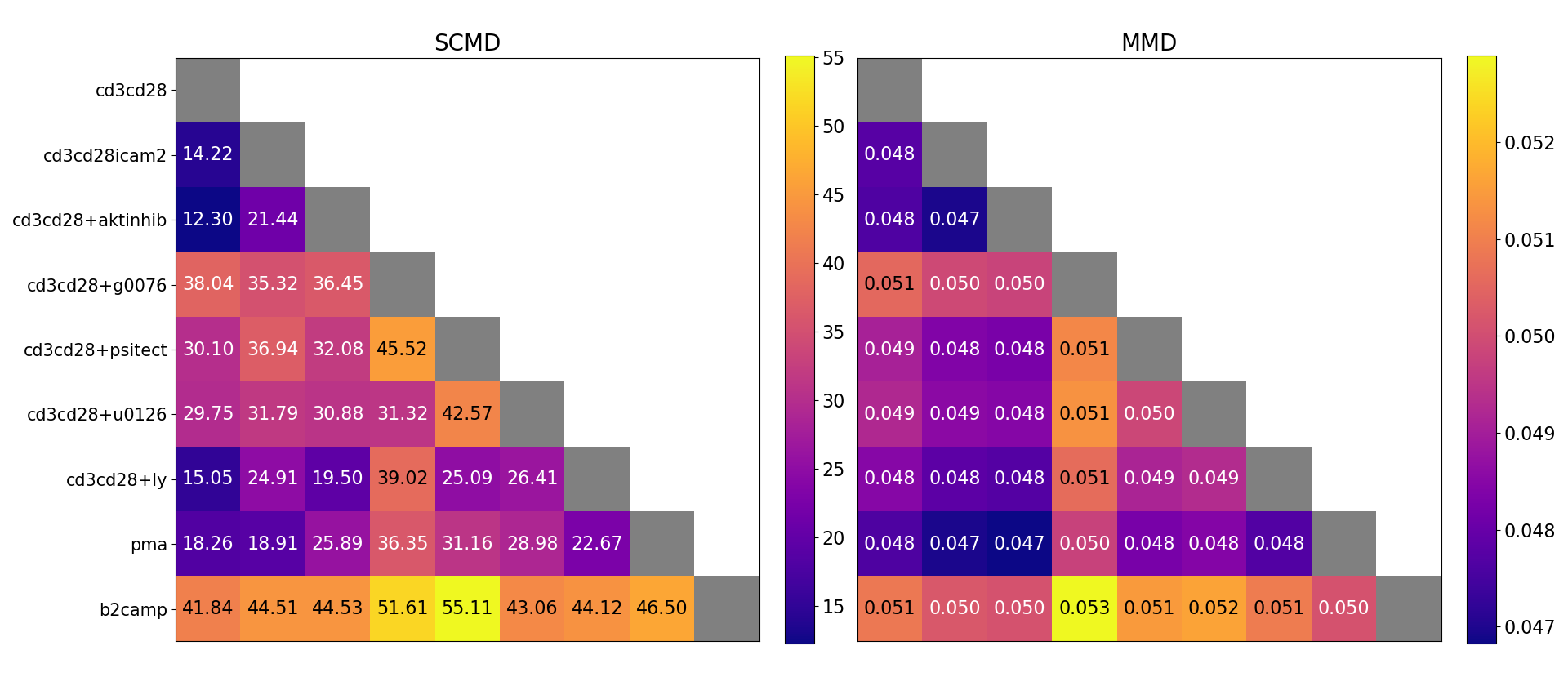}
    \caption{Heatmaps of pairwise environment distances (lower-triangular part shown). Left: SCMD, capturing both structural and distributional discrepancies. Right: MMD, based only on distributional differences.}
    \label{fig:SCMD&MMD2}
\end{figure*}

\paragraph{Result analysis}
The average results over 50 repetitions are presented in Table~\ref{tab:res}. As discussed in Section~\ref{sec:ex}, MMD and SID each fail to detect differences in one of the two scenarios. 

SCMD successfully quantifies both types of discrepancies.
In Case 1, SCMD reports a  distance of 0.5177, reflecting the parametric mismatch between models.
In Case 2, SCMD yields an even larger value (0.8921), indicating a fundamental structural divergence.
The plug-in estimator of SCMD exhibits low variance, as expected in this controlled parametric setting.
The kernel-based estimator shows slightly higher variance but remains consistent with theoretical values, validating its robustness in nonparametric contexts.

When a prediction task is of interest, P-SCMD is more relevant, as it  disentangles the differences with respect to the nodes. For instance, when predicting $X$, the two datasets are almost indistinguishable in Case~1 (where only $Y$ differs), whereas both prediction tasks are challenging in Case~2 (where the causal structure differs). Notably, SCMD can be interpreted as the sum of the P-SCMD values over all variables.

To avoid manually selecting intervention values for each variable, E-SCMD computes an expectation over interventions. 
E-SCMD (0.5819 vs. 0.9473) corroborates SCMD trends while eliminating dependency on manual intervention choices.


\subsection{Real-world dataset}
We analyze the dataset introduced by \citet{sachs2005causal}, consisting of measurements of 11 phosphorylated proteins and phospholipids across 7,466 immune system cells. The data were collected under 9 molecular interventions using multiparameter flow cytometry. The causal relationships among the proteins, obtained via causal discovery and validated by experts, are shown in Figure~\ref{fig:Sachs}.  
Each intervention corresponds to a distinct environment, although the interventions do not directly target the 11 observed variables (see Figure 2 in \citealp{sachs2005causal}). We illustrate how SCMD can quantify the pairwise distances between these 9 environments. In this setting, the causal graph remains fixed, and we therefore focus on changes in the interventional distributions. For comparison, we additionally compute the MMD between the joint distributions of each pair of environments. We use $\sigma^2 = 10$, $\lambda = 1$, and $\bm{v}$ is set to the mean of each variable in each environment. 

The heatmaps in Figure \ref{fig:SCMD&MMD2} show the pairwise distances between environments, with SCMD on the left and MMD on the right. Environment labels correspond to the interventions applied. Only the lower-triangular part of the matrices is displayed, as the distances are symmetric.  
For SCMD, all environments except the last two involve interventions on $CD3$ and $CD28$, which are ancestors of the first variables in Figure \ref{fig:Sachs}. Environments 'cd3cd28' and 'cd3cd28+aktinhib' are highly similar, since the 'aktinhib' intervention targets \texttt{Akt}, located at the end of the causal order, causing only minor changes in the interventional distributions. By contrast, environments 'cd3cd28+psitect' and 'cd3cd28+u0126' differ substantially: the interventions 'psitect' and 'u0126' target \texttt{PIP2} and \texttt{Mek}, near the beginning of the causal order, producing large distributional shifts. Additionally, the SCM associated with the “b2camp” intervention shows marked differences from all others, likely due to stronger intervention effects and unobserved links between variables \citep{10.1189/jlb.2RU1116-474R}.  
Comparing SCMD with MMD highlights the added granularity of SCMD. While SCMD reflects the effect of interventions depending on their position in the causal structure, MMD fails to capture this: many pairs of environments appear equally distant under MMD, even when SCMD detects substantial differences.

\section{Conclusion and discussion} \label{sec:conc}
This work introduced a kernel-based framework to quantify distances between Structural Causal Models (SCMs), enabling rigorous comparison of interventional distributions. Our main contribution is a theoretically grounded metric that accounts for both structural and parametric differences, with estimation procedure guarantees on identifiability and robustness.

Our approach assumes known causal graphs, which may not hold in practice. Discovering it in a first step would induce a bias to take into account in the procedure. Estimation also relies on conditioning arguments within the observed data support; when this is not feasible (e.g., for extreme values of interventions), results may become unreliable, highlighting the need for robust extrapolation or regularization techniques.

This distance can be used to quantify the difficulty of a trained model to predict on a distribution out-of-domain. Future research could explore whether, as with MMD, our distance can be leveraged to design a test of equality between SCMs. \\
More broadly, it offers a  way to characterize  distributional shift by measuring how far causal mechanisms deviate from those seen during training. This perspective opens new directions for adaptive strategies: for instance, selecting models or regularization schemes according to  measured difficulty, or designing targeted interventions when large structural mismatches are detected. Ultimately, our metric could thus serve as a foundation for a more systematic approach to robustness in out-of-distribution generalization.

\bibliography{biblio}


\appendix
\section{Theoretical details}
In this section, we provide proofs of theoretical results. 
\subsection{Properties of SCMD: proofs}
\begin{proposition}[SCMD is a distance]
Let $\mathcal{M}^1,\mathcal{M}^2$ two SCMs and $(\bm{v}^1, \bm{v}^2)$ two vectors of intervention values. If the kernel $k_{\mathcal{V}}$ is characteristic, then $\mathrm{SCMD}$ is a distance: 
    \begin{itemize}
        \item  \textbf{Separation:} $\mathrm{SCMD}(\mathcal{M}^1,\mathcal{M}^2 ; \bm{v}^1, \bm{v}^2) = 0 \Leftrightarrow \mathcal{M}^1 = \mathcal{M}^2$ almost surely;
        \item  \textbf{Non-negativity:}  $\mathrm{SCMD}(\mathcal{M}^1,\mathcal{M}^2; \bm{v}^1, \bm{v}^2) \geq 0$;
        \item \textbf{Symmetry:} $$\mathrm{SCMD}(\mathcal{M}^1,\mathcal{M}^2; \bm{v}^1, \bm{v}^2) = \mathrm{SCMD}(\mathcal{M}^2,\mathcal{M}^1; \bm{v}^2, \bm{v}^1);$$
        \item \textbf{Triangle inequality:} for any third SCM $\mathcal{M}_3$,  \begin{align*}
            \mathrm{SCMD}(\mathcal{M}^1,\mathcal{M}^2; \bm{v}^1, \bm{v}^2)
            \leq \mathrm{SCMD}(\mathcal{M}^1,\mathcal{M}^3; \bm{v}^1, \bm{v}^3) +\mathrm{SCMD}(\mathcal{M}^3,\mathcal{M}^2; \bm{v}^3, \bm{v}^2).
        \end{align*}
    \end{itemize}
\end{proposition}

\begin{proof}
SCMD is a sum of MIMD terms, which are metrics. Then, non-negativity, symmetry and triangle inequality directly follows because MIMD is non-negative, symmetric and satisfies the triangle inequality. 
Thus, we just have to prove the separation.

    \textbf{Separation.} 
     $(\Leftarrow)$ If the two structural causal models are identical, then all discrepancies between their interventional distributions necessarily vanish. Thus, $\mathrm{SCMD}(\mathcal{M}^1,\mathcal{M}^2 ; \bm{v}^1, \bm{v}^2) = 0$. 
     
     $(\Rightarrow)$ We assume that $\mathrm{SCMD}(\mathcal{M}^1,\mathcal{M}^2 ; \bm{v}^1, \bm{v}^2) = 0$. Since the SCMD is defined as the sum of MIMDs that are positive, it leads to, for every pair $(i,j)$,
     $\mathrm{MIMD}(P^1_{\mathrm{do}(V_i=v_i)}(V_j), P^2_{\mathrm{do}(V_i = v_i)}(V_j)) = 0$. 
     From  \citet[Theorem 5.2]{park2020measure}, it means that 
     $$P^1_{\mathrm{do}(V_i=v_i)}(B) = P^2_{\mathrm{do}(V_i=v_i)}(B)$$
     for all $B$, because the kernel is characteristic and we assume that each distribution is absolutely continuous with respect to each other. 
     
This means that the two SCMs $\mathcal{M}_1$ and $\mathcal{M}_2$ are interventional equivalent \citep[Definition 4.3]{bongers2021foundations}, which we wrote, by a slight abuse of notations, $\mathcal{M}_1 = \mathcal{M}_2$.
\end{proof}

\begin{proposition}[Bounds and relation to SID]
Let $\mathcal{M}^1,\mathcal{M}^2$ two SCMs and $\bm{v}^1, \bm{v}^2$ two vectors of intervention values. 
\begin{enumerate}
\item \textbf{Boundedness:} if the kernel $k_{\mathcal{V}}$ is bounded, then SCMD is bounded as well:
     if for all $(v_j, v'_j) \in \mathcal{V}_j^2$, $k_{\mathcal{V}_j}(v_j,v'_j)\leq K$ for $K\in \mathbb{R}^+$, then 
 $$\mathrm{SCMD}(\mathcal{M}^1,\mathcal{M}^2 ; \bm{v}^1, \bm{v}^2)\leq \sqrt{2K}d(d-1).$$

\item \textbf{Relation to SID:} 
\begin{align*}
    \mathrm{SCMD}(\mathcal{M}^1,\mathcal{M}^2; \bm{v}^1, \bm{v}^2)=0 \Rightarrow 
     \mathrm{SID}(\mathcal{G}^1,\mathcal{G}^2)= \mathrm{SID}(\mathcal{G}^2,\mathcal{G}^1)=0. 
\end{align*}
\end{enumerate}
\end{proposition}

\begin{proof}
    \textbf{1. Boundedness.} Since $k_{\mathcal{V}_j}(v_j,v'_j)\leq K$ with $K\in\mathbb{R}^+$ for all $(v_j, v'_j) \in \mathcal{V}_j^2$, the interventional (conditional or mean) embedding is also bounded:
    \begin{align*}
        \|\mu_{P_{\mathrm{do}(V_i = v_i)}(V_j)}(\cdot)\|_{\mathcal{H}_{\mathcal{V}_j}}  
        &= \| \mathbb{E}_{\mathbf{Z}\sim P(\mathbb{Z})} [\mu_{P(V_j|V_i=v_i, \mathbf{Z}= \mathbf{z})}(\cdot)] \|_{\mathcal{H}_{\mathcal{V}_j}} \\
        &\leq \sqrt{\mathbb{E}_{\mathbf{Z}\sim P(\mathbb{Z})} [\mathbb{E}_{V_j,V'_j \sim P(\cdot | V_i=v_i, \mathbf{Z})} [k_{\mathcal{V}_j}(V_j,V'_j)]]}\\
        &\leq \sqrt{\|k_{\mathcal{V}_k}\|_\infty} =  \sqrt{K}.
    \end{align*}
    
    Therefore, 
    \begin{align*}
     \mathrm{MIMD}(P^1_{\mathrm{do}(V_i=v_i)}(V_j), P^2_{\mathrm{do}(V_i = v_i)}(V_j))
    := &\|\mu_{P^1_{\mathrm{do}(V_i = v_i)}(V_j)}(\cdot) - \mu_{P^2_{\mathrm{do}(V_i = v_i)}(V_j)}(\cdot)\|_{\mathcal{H}_{\mathcal{V}_j}}\\
    \leq & \sqrt{ \|\mu_{P^1_{\mathrm{do}(V_i = v_i)}(V_j)}(\cdot)\|_{\mathcal{H}_{\mathcal{V}_j}}^2  + \| \mu_{P^2_{\mathrm{do}(V_i = v_i)}(V_j)}(\cdot)\|_{\mathcal{H}_{\mathcal{V}_j}}^2} \leq \sqrt{2K}.
    \end{align*}
    
    Since SCMD is the sum of $d(d-1)$  MIMD terms, is computed $d(d-1)$ times, this establishes the desired bound for the SCMD.

    \textbf{2. Relation to SID.} Suppose that $\mathrm{SCMD}(\mathcal{M}^1,\mathcal{M}^2; \bm{v}^1, \bm{v}^2)=0$. 
   From Proposition 1, it means that $\mathcal{M}^1$ and $\mathcal{M}^2$ are interventionally equivalent. 
   From \citet{hauser2012}, it means that  the skeleton and the V-structure in both models are the same. 
   As we consider every direct effect, computing interventional distribution for every pair of variables, it means that the two graphs coincide: $\mathcal{G}^1 = \mathcal{G}^2$.
 Consequently, $$\mathrm{SID}(\mathcal{G}^1,\mathcal{G}^2)=\mathrm{SID}(\mathcal{G}^2,\mathcal{G}^1)=0.$$
\end{proof}

\subsection{Illustrative example: formula in case 2}
The main paper includes the explicit SCMD formula for Case 1. Below is the formula for Case 2, omitted in the main text for brevity:
\begin{align*}
   & \text{SCMD}(\mathcal{M}^{1,a}, \mathcal{M}^{2,a}; (x, y) )
   =\|\mu_{\mathcal{N}(a x,\,1)} - \mu_{\mathcal{N}\!\left(0,\,1+a^2\right)}\|_{\mathcal{H}}
   + \|\mu_{\mathcal{N}(0,\,1)} - \mu_{ \mathcal{N}\!\left(\tfrac{a}{1+a^2}y,\,\tfrac{1}{1+a^2}\right)}\|_{\mathcal{H}}\\
   =& \left( \sqrt{\frac{\sigma^2}{\sigma^2+2}} + \sqrt{\frac{\sigma^2}{\sigma^2+2(1+a^2)}}  
   - 2\sqrt{\frac{\sigma^2}{\sigma^2+2+a^2}}
     \exp{\left(-\frac{a^2x^2}{2(\sigma^2+2+a^2)}\right)}\right)^{1/2} \\
   &+ \left( \sqrt{\frac{\sigma^2}{\sigma^2+2}}  + \sqrt{\frac{\sigma^2}{\sigma^2+\frac{2}{1+a^2}}}  -2 \sqrt{\frac{\sigma^2}{\sigma^2+1+\frac{1}{1+a^2}}}
     \exp{\left(-\frac{\frac{a^2}{1+a^2}y^2}{2(\sigma^2+1+\frac{1}{1+a^2})}\right)} \right)^{1/2}
\end{align*}

\subsection{Study of the estimator: proof of the rate of convergence}

\begin{theorem}
Suppose that $k_{\mathcal{V}_i}, k_{\mathcal{V}_j}$ and $k_{\mathcal{Z}}$ are bounded kernels, and that the operator-valued kernel $k_{\mathcal{V}\mathcal{Z}}$ is $C_0$ universal. Let the regularization parameter $\lambda$ decay to 0 at a slower rate than $\mathcal{O}(N^{-1/2})$. 
Then $\widehat{\mu}^{\lambda,N}_{P_{\text{do}(V_i = v_i)}(V_j)}$ is universally consistent.

Assume further that $\mu_{P_{\text{do}(V_i = v_i)}(V_j)}(\cdot) \in \widetilde{\mathcal{H}}$. Then, with probability $1-\delta$, 
$$\tilde{\mathcal{E}}_{N, \lambda}(\widehat{\mu}^{\lambda,N}_{P_{\text{do}(V_i = v_i)}(V_j)}) - \tilde{\mathcal{E}}_{ N, \lambda}(\mu_{P_{\text{do}(V_i = v_i)}(V_j)}) = \mathcal{O}_P (N^{-1/4}).$$
\end{theorem}

\begin{proof}
\begin{itemize}
    \item \textbf{Universal consistency} It directly follows from \cite{park2020,park2020measure}.
Specifically, the interventional embedding is constructed as the empirical mean, over the adjustment set $\mathbf{Z}$, of the conditional mean embedding.
 \citet{park2020,park2020measure} establish that the estimator  $\widehat{\mu}^{\lambda,N}_{P(V_j | V_i = v_i, \mathbf{Z})}$ is universally consistent. If the kernel is universal and $\mathbf{Z}$ compact, $\widehat{\mu}^{\lambda,N}_{P(V_j | V_i = v_i, \mathbf{Z})}$ converges  uniformly on $\mathbf{Z}$.
   Therefore, the resulting estimator  $\mathbb{E}_{\mathbf{Z}}(\widehat{\mu}^{\lambda,N}_{P(V_j | V_i = v_i, \mathbf{Z})})$ is universally consistent. 
   \item \textbf{Rate of convergence.} Similarly, as the interventional estimator is an empirical mean of a conditional mean embedding, the rate of convergence follows from  \citet{park2020,park2020measure}.
\end{itemize}
\end{proof}

\section{Algorithmic details}
\subsection{Python code notice}
We provide a zip archive containing three Python scripts to ensure full reproducibility of our experiments:
\begin{itemize}
    \item \texttt{SCMD\_Function.py} implements core functions: kernel, norm, cross-norm, and the SCMD algorithm.
    \item \texttt{Synthetic\_data\_experiments.py} reproduces all experiments on synthetic data, as described in the main paper and appendix. {This file is organized as follows. The first two for loops reproduce the results corresponding to Case 1 and Case 2, respectively. Next, the theoretical values of SCMD and MMD are computed. The final section contains the code used to generate Figure 1 in the main paper.}
    \item \texttt{Real-world\_data\_experiments.py} reproduces all experiments on real-world datasets, as described in the main paper and appendix.  {This file is organized as follows. After the code generating the datasets, the subsequent sections contain the loops used for the computation of SCMD and MMD. The final part includes the code employed to produce Figure 3 in the main paper.}
\end{itemize}

\subsection{Algorithmic complexity}
In \textbf{Algorithm~1}, we compute the estimation of the interventional/conditional/mean embedding. This relies, in the worst case, on the computation of a Gram matrix associated with a kernel and on the inversion of a matrix. Each computation has a complexity of $\mathcal{O}(N^3)$ for a sample size of $N$, but we have at most $d$ Gram matrices to compute (hidden in the Hadamard product). Therefore, the worst-case complexity of \textbf{Algorithm~1} is $\mathcal{O}(dN^3)$.

\textbf{Algorithm~2} also computes Gram matrices for one variable and uses \textbf{Algorithm~1}. Thus, its complexity is also $\mathcal{O}(dN^3)$.

To compute the SCMD, we apply MIMD to every pair of variables, i.e., $d(d-1)$ times. The total computational complexity then becomes $\mathcal{O}(d^3N^3)$.

As an illustration, we display the  computation time of SCMD on a standard laptop.
For a dataset with $d=5$ variables and $N=1000$ observations, it  requires approximately $7$ seconds, for a dataset with $d=5$ variables and $N=10000$ observations, it increases to about $12$ minutes on the same hardware.

\section{Additional experiments}

\subsection{Illustrative example: sensitivity analysis}
We conduct additional experiments to evaluate the impact of the two hyperparameters: $\sigma^2$ (the kernel bandwidth) and $\lambda$ (the regularization parameter in the Ridge regression). The corresponding results are reported in Table \ref{tab:res_sigma01}. 

\textbf{Regarding $\lambda$.} As expected, increasing $\lambda$ leads to 
a slight decrease in the values  of $\widehat{\mathrm{SCMD}}_{\mathrm{plug\text{-}in}}$, and consequently in  those of $\widehat{\mathrm{P\text{-}SCMD}}_{Y}$ and $\widehat{\mathrm{P\text{-}SCMD}}_{X}$. This effect is more pronounced in Case 1, indicating a higher sensitivity to the regularization  parameter $\lambda$. 

\textbf{Regarding $\sigma^2$.} The theoretical formulation of SCMD explicitly depends on $\sigma^2$, and our results reflect this dependency. In Case 1, the estimated values range from $0.5177$ ($\sigma^2 = 0.1$) to 
$0.7549$ ($\sigma^2 = 1.5$), while in Case 2, they vary from $0.8921$ to $1.0048$. Notably, there is no clear monotonic trend: in Case 2, the highest value is observed for $\sigma^2 = 1$, suggesting that this bandwidth captures more relevant information for the given data. 

This underscores the importance of careful tuning of $\sigma^2$, as widely recognized in the literature (with several strategies to select it, for example as in \citet{biggs2023mmdfuse}).
\begin{table}[t]
\centering
\caption{Results on the simulated example.
Case 1: same causal graph but different joint distributions;
Case 2: different causal graphs but same joint distribution.
Reported values are means $\pm$ standard deviation over $50$ repetitions. We vary the kernel bandwidth (each subtable), and the regularization parameter for SCMD.}
\label{tab:res_sigma01}
\subcaption{$\sigma^2 = 0.1$. For Case 1, SID = 0, MMD = 0.0515, its estimation is $0.1215 \pm 0.0013$; for Case 2, SID = 2, MMD = 0, its estimation is $0.0140 \pm 0.0016$. }
\begin{tabular}{lcccccc}
\toprule
 & \multicolumn{3}{c}{Case 1} & \multicolumn{3}{c}{Case 2} \\
\midrule
SCMD & \multicolumn{3}{c}{$0.5177$} & \multicolumn{3}{c}{$0.8921$} \\
\midrule
$\widehat{\mathrm{SCMD}}_{\mathrm{plug\text{-}in}}$ & \multicolumn{3}{c}{$0.5248 \pm 0.005$} & \multicolumn{3}{c}{$0.8929 \pm 0.004$} \\
\midrule
& $\lambda=0.1$ & $\lambda=0.5$ & $\lambda=1$ & $\lambda=0.1$ & $\lambda=0.5$ & $\lambda=1$ \\ 
\midrule
$\widehat{\mathrm{SCMD}}_{\mathrm{kernel}}$ & $0.5369 \pm 0.02$ & $0.5360 \pm 0.02$ & $0.5351 \pm 0.02$ & $1.0053 \pm 0.02$ & $1.0046 \pm 0.02$ & $1.0038 \pm 0.02$ \\
\midrule
$\widehat{\mathrm{P\text{-}SCMD}}_{Y}$ & $0.5247 \pm 0.02$ & $0.5237 \pm 0.02$ & $0.4059 \pm 0.01$ & $0.4066 \pm 0.01$ & $0.4062 \pm 0.01$ & $0.4059 \pm 0.01$ \\
$\widehat{\mathrm{P\text{-}SCMD}}_{X}$ & $0.0122 \pm 0.00$ & $0.0122 \pm 0.00$ & $0.0122 \pm 0.00$ & $0.5987 \pm 0.02$ & $0.5984 \pm 0.01$ & $0.5979 \pm 0.01$ \\
\midrule
$\widehat{\mathrm{E\text{-}SCMD}}$ & $0.5846 \pm 0.01$ & $0.5819 \pm 0.01$ & $0.5797 \pm 0.01$ & $0.9469 \pm 0.01$ & $0.9473 \pm 0.01$ & $0.9454 \pm 0.01$ \\
\bottomrule
\end{tabular}
\subcaption{$\sigma^2 = 1$. For Case 1, SID = 0, MMD = 0.1758, its estimation is $0.1761 \pm 0.004$; for Case 2, SID = 2, MMD = 0, its estimation is $0.0128 \pm 0.004$.}
\begin{tabular}{lcccccc}
\toprule
 & \multicolumn{3}{c}{Case 1} & \multicolumn{3}{c}{Case 2} \\
\midrule
SCMD & \multicolumn{3}{c}{$0.7496$} & \multicolumn{3}{c}{$1.0048$} \\
\midrule
$\widehat{\mathrm{SCMD}}_{\mathrm{plug\text{-}in}}$ & \multicolumn{3}{c}{$0.7588 \pm 0.01$} & \multicolumn{3}{c}{$1.0050 \pm 0.00$} \\
\midrule
& $\lambda=0.1$ & $\lambda=0.5$ & $\lambda=1$ & $\lambda=0.1$ & $\lambda=0.5$ & $\lambda=1$ \\ 
\midrule
$\widehat{\mathrm{SCMD}}_{\mathrm{kernel}}$ & $0.6433 \pm 0.01$ & $0.6099 \pm 0.01$ & $0.5925 \pm 0.01$ & $1.2035 \pm 0.01$ & $1.1991 \pm 0.01$ & $1.1961 \pm 0.01$ \\
\midrule
$\widehat{\mathrm{P\text{-}SCMD}}_{Y}$ & $0.6354 \pm 0.01$ & $0.6021 \pm 0.01$ & $0.5847 \pm 0.01$ & $0.6257 \pm 0.01$ & $0.6232 \pm 0.01$ & $0.6202 \pm 0.01$ \\
$\widehat{\mathrm{P\text{-}SCMD}}_{X}$ & $0.0078 \pm 0.00$ & $0.0078 \pm 0.00$ & $0.0078 \pm 0.00$ & $0.5760 \pm 0.01$ & $0.5759 \pm 0.01$ & $0.5759 \pm 0.01$ \\
\midrule
$\widehat{\mathrm{E\text{-}SCMD}}$ & $0.9069 \pm 0.01$ & $0.8847 \pm 0.01$ & $0.8712 \pm 0.01$ & $1.0968 \pm 0.01$ & $1.0917 \pm 0.01$ & $1.0882 \pm 0.01$ \\
\bottomrule
\end{tabular}
\subcaption{$\sigma^2 = 1.5$. For Case 1, SID = 0, MMD = 0.1941, its estimation is $0.1850 \pm 0.004$; for Case 2, SID = 2, MMD = 0, its estimation is $0.0123 \pm 0.004$.}
\begin{tabular}{lcccccc}
\toprule
 & \multicolumn{3}{c}{Case 1} & \multicolumn{3}{c}{Case 2} \\
\midrule
SCMD & \multicolumn{3}{c}{$0.7549$} & \multicolumn{3}{c}{$0.9776$} \\
\midrule
$\widehat{\mathrm{SCMD}}_{\mathrm{plug\text{-}in}}$ & \multicolumn{3}{c}{$0.7639 \pm 0.01$} & \multicolumn{3}{c}{$0.9776 \pm 0.00$} \\
\midrule
& $\lambda=0.1$ & $\lambda=0.5$ & $\lambda=1$ & $\lambda=0.1$ & $\lambda=0.5$ & $\lambda=1$ \\ 
\midrule
$\widehat{\mathrm{SCMD}}_{\mathrm{kernel}}$ & $0.6028 \pm 0.01$ & $0.5671 \pm 0.01$ & $0.5471 \pm 0.01$ & $1.1816 \pm 0.01$ & $1.1750 \pm 0.01$ & $1.1691 \pm 0.01$ \\
\midrule
$\widehat{\mathrm{P\text{-}SCMD}}_{Y}$ & $0.5959 \pm 0.01$ & $0.5602 \pm 0.01$ & $0.5403 \pm 0.01$ & $0.6511 \pm 0.01$ & $0.6443 \pm 0.01$ & $0.6381 \pm 0.01$ \\
$\widehat{\mathrm{P\text{-}SCMD}}_{X}$ & $0.0069 \pm 0.00$ & $0.0069 \pm 0.00$ & $0.0068 \pm 0.00$ & $0.5304 \pm 0.01$ & $0.5307 \pm 0.01$ & $0.5310 \pm 0.01$ \\
\midrule
$\widehat{\mathrm{E\text{-}SCMD}}$ & $0.8847 \pm 0.01$ & $0.9018 \pm 0.01$ & $0.8932 \pm 0.00$ & $1.0599 \pm 0.01$ & $1.0515 \pm 0.01$ & $1.0460 \pm 0.01$ \\
\bottomrule
\end{tabular}
\end{table}

\clearpage
\subsection{Real-world dataset: sensitivity analysis}
In Figure \ref{fig:SCMD_MMD_comparison}, we display the heatmaps of SCMD and MMD (respectively left and right) computed for each pair of environments under varying values of $\sigma^2$. The environments 'cd3cd28' and 'cd3cd28+aktinhib' remain highly similar according to SCMD, regardless of the value of $\sigma^2$. Conversely, the "b2camp" intervention continues to exhibit high dissimilarity with all other environments. The MMD, in contrast, appears unaffected by the considered variations in $\sigma^2$.

\begin{figure*}[t]
    \centering
    \begin{subfigure}[t]{\textwidth}
        \centering
        \includegraphics[width=0.8\textwidth, trim = 0 1.5cm 0 0.9cm, clip = TRUE]{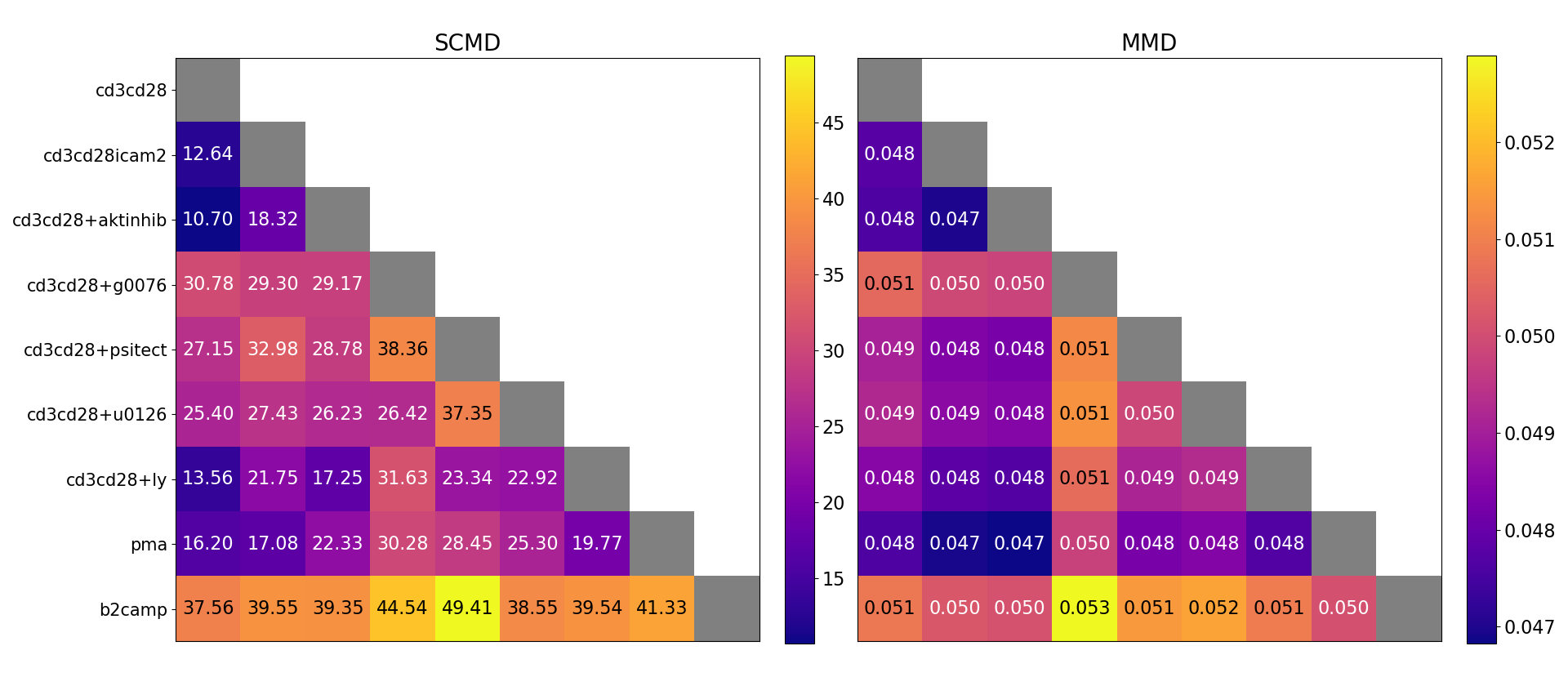}
        \caption{$\sigma^2=5$}
        \label{fig:sigma5}
    \end{subfigure}\vspace{0.5em}
    
    \begin{subfigure}[t]{\textwidth}
        \centering
        \includegraphics[width=0.8\textwidth, trim = 0 1.5cm 0 0.9cm, clip = TRUE]{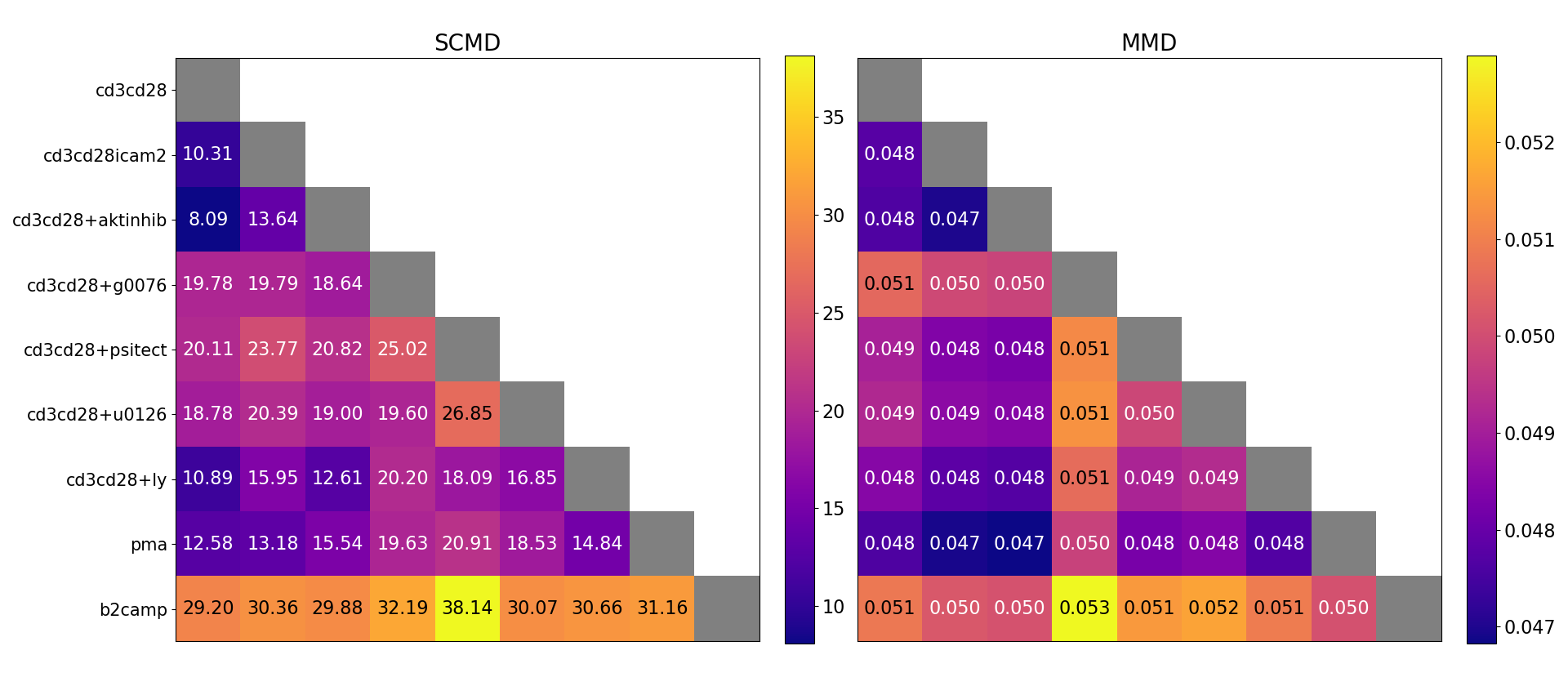}
        \caption{$\sigma^2=1$}
        \label{fig:sigma1}
    \end{subfigure}
\caption{Heatmaps of pairwise environment distances (lower-triangular part shown). Left: SCMD, capturing both structural and distributional discrepancies. Right: MMD, based only on distributional differences. We test two values for the kernel bandwidth: top, (a) $\sigma^2=5$, bottom, (b) $\sigma^2=1$.}
    \label{fig:SCMD_MMD_comparison}
\end{figure*}

\end{document}